\documentclass{amsproc}%
\usepackage{amsfonts}%
\usepackage{amsmath}%
\usepackage{amssymb}%
\usepackage{graphicx}
\theoremstyle{plain}

\newtheorem{lemma}{Lemma}

\newtheorem{proposition}{Proposition}

\newtheorem{theorem}{Theorem}
\numberwithin{equation}{section}
\begin{document}
\title[Products of Beta  variables]{Products of Beta distributed random variables}
\author{Charles F. Dunkl}
\address{Dept. of Mathematics, PO Box 400137\\
University of Virginia, Charlottesville VA 22904-4137}
\email{cfd5z@virginia.edu}
\urladdr{http://people.virginia.edu/\symbol{126}cfd5z/home.html}
\date{23 April 2013}

\begin{abstract}
This is an expository note on useful expressions for the density function of a
product of independent random variables where each variable has a Beta distribution.

\end{abstract}
\maketitle

\section{Introduction}

This is a brief exposition of some techniques to construct density functions
with moment sequences of the form $\prod\limits_{j=1}^{m}\dfrac{\left(
u_{j}\right)  _{n}}{\left(  v_{j}\right)  _{n}}$, where $\left(  a\right)
_{n}$ denotes the Pochhammer symbol $\frac{\Gamma\left(  a+n\right)  }%
{\Gamma\left(  a\right)  }$. Such a density $f\left(  x\right)  $ can be
expressed as a certain Meijer G-function, that is, a sum of generalized
hypergeometric series, and as a power series in $\left(  1-x\right)  $ whose
coefficients can be calculated by a recurrence. The former expression is
pertinent for numerical computations for $x$ near zero, while the latter is
useful for $x$ near $1$.

All the random variables considered here take values in $\left[  0,1\right]
$, density functions are determined by their moments: for a random variable
$X$ we have $P\left[  X<a\right]  =\int_{0}^{a}f\left(  x\right)  dx$ for
$0\leq a\leq1$, and the expected value $E\left(  X^{n}\right)  =\int_{0}%
^{1}x^{n}f\left(  x\right)  dx$ is the $n$th moment. The basic building block
is the Beta distribution ($\alpha,\beta>0$)
\begin{equation}
h\left(  \alpha,\beta;x\right)  =\frac{1}{B\left(  \alpha,\beta\right)
}x^{\alpha-1}\left(  1-x\right)  ^{\beta-1},\label{fbeta}%
\end{equation}
where $B\left(  \alpha,\beta\right)  :=\frac{\Gamma\left(  \alpha\right)
\Gamma\left(  \beta\right)  }{\Gamma\left(  \alpha+\beta\right)  }$, then%
\[
\int_{0}^{1}x^{n}h\left(  \alpha,\beta;x\right)  dx=\frac{\left(
\alpha\right)  _{n}}{\left(  \alpha+\beta\right)  _{n}},n=0,1,2,\ldots.,
\]
thus $\left\{  \frac{\left(  u\right)  _{n}}{\left(  v\right)  _{n}}\right\}
$ is a moment sequence if $0<u<v$ (with $\alpha=u,\beta=v-u$). The moments of
the product of independent random variables are the products of the respective
moments, that is, suppose the densities of (independent) $X$ and $Y$ are $f,g$
respectively and define%
\begin{equation}
f\ast g\left(  x\right)  =\int_{x}^{1}f\left(  t\right)  g\left(  \frac{x}%
{t}\right)  \frac{dt}{t},\label{convol}%
\end{equation}
then $f\ast g$ is a density, $P\left[  XY<a\right]  =\int_{0}^{a}f\ast
g\left(  x\right)  dx$ for $0\leq a\leq1$ and%
\[
\int_{0}^{1}x^{n}\left(  f\ast g\left(  x\right)  \right)  dx=\int_{0}%
^{1}x^{n}f\left(  x\right)  dx\int_{0}^{1}y^{n}g\left(  y\right)  dy.
\]
These are the main results: suppose the parameters $u_{1},\ldots,u_{m}$ and
$v_{1},\ldots,v_{m}$ satisfy $v_{i}>u_{i}>0$ for each $i$, then there is a
unique density function $f$ with the moment sequence $\prod\limits_{j=1}%
^{m}\dfrac{\left(  u_{j}\right)  _{n}}{\left(  v_{j}\right)  _{n}}$;

\begin{enumerate}
\item if also $u_{i}-u_{j}\notin\mathbb{Z}$ for each $i\neq j$ then for $0\leq
x<1$%
\begin{align}
f\left(  x\right)   &  =\left(  \prod_{k=1}^{m}\frac{\Gamma\left(
v_{k}\right)  }{\Gamma\left(  u_{k}\right)  }\right)  \sum_{i=1}^{m}\frac
{1}{\Gamma\left(  v_{i}-u_{i}\right)  }\prod_{j=1,j\neq i}^{m}\frac
{\Gamma\left(  u_{j}-u_{i}\right)  }{\Gamma\left(  v_{j}-u_{i}\right)
}x^{u_{i}-1}\label{bigf}\\
&  \times\sum_{n=0}^{\infty}\prod_{k=1}^{m}\frac{\left(  u_{i}-v_{k}+1\right)
_{n}}{\left(  u_{i}-u_{k}+1\right)  _{n}}x^{n};\nonumber
\end{align}

\item for $\delta:=\sum_{i=1}^{m}\left(  v_{i}-u_{i}\right)  $ there is an
$\left(  m+1\right)  $-term recurrence for the coefficients $\left\{
c_{n}\right\}  $ such that%
\begin{equation}
f\left(  x\right)  =\frac{1}{\Gamma\left(  \delta\right)  }\prod_{i=1}%
^{m}\frac{\Gamma\left(  v_{i}\right)  }{\Gamma\left(  u_{i}\right)  }\left(
1-x\right)  ^{\delta-1}\left\{  1+\sum_{n=1}^{\infty}c_{n}\left(  1-x\right)
^{n}\right\}  ,0<x\leq1.\label{recur1}%
\end{equation}

\end{enumerate}

The use of the inverse Mellin transform to derive the series expansion in
(\ref{bigf}) is sketched in Section 2. The differential equation initial value
problem for the density is described in Section 3, and the recurrence for
(\ref{recur1}) is derived in Section 4.

The examples in Section 5 include the relatively straightforward situation
$m=2$ and the density of the determinant of a random $4\times4$
positive-definite matrix of trace one, where $m=3$.

\section{The inverse Mellin transform}

The Mellin transform of the density $f$ is defined by%
\[
Mf\left(  p\right)  =\int_{0}^{1}x^{p-1}f\left(  x\right)  dx.
\]
This is an analytic function in $\left\{  p:\operatorname{Re}p>0\right\}  $
and agrees with the meromorphic function%
\[
p\mapsto\prod_{j=1}^{m}\frac{\Gamma\left(  v_{j}\right)  \Gamma\left(
u_{j}+p-1\right)  }{\Gamma\left(  u_{j}\right)  \Gamma\left(  v_{j}%
+p-1\right)  }%
\]
at $p=1,2,3,\ldots$ thus the two functions coincide in the half-plane by
Carlson's theorem. The inverse Mellin transform is%
\[
f\left(  x\right)  =\frac{1}{2\pi i}\int_{\sigma-i\infty}^{\sigma+i\infty
}Mf\left(  p\right)  x^{-p}dp,
\]
for $\sigma>0$; it turns out the integral can be evaluated by residues (it is
of Mellin-Barnes type). For each $j$ and each $n=0,1,2,\ldots$ there is a pole
of $Mf\left(  p\right)  $ at $p=1-n-u_{j}$; the hypothesis $u_{i}-u_{j}%
\notin\mathbb{Z}$ for each $i\neq j$ implies that each pole is simple. The
residue at $p=1-n-u_{k}$ equals%
\begin{align*}
&  x^{u_{k}-1+n}\prod_{j=1}^{m}\frac{\Gamma\left(  v_{j}\right)  }%
{\Gamma\left(  u_{j}\right)  \Gamma\left(  v_{j}-u_{k}-n\right)  }\prod_{i\neq
k}\Gamma\left(  u_{i}-u_{k}-n\right)  \\
&  \times\lim_{p\rightarrow1-n-u_{k}}\left(  p-1+n+u_{k}\right)  \Gamma\left(
u_{k}+p-1\right)  .
\end{align*}
To simplify this we use%
\begin{align*}
\Gamma\left(  a-n\right)   &  =\frac{\Gamma\left(  a\right)  }{\left(
a-n\right)  _{n}}=\left(  -1\right)  ^{n}\frac{\Gamma\left(  a\right)
}{\left(  1-a\right)  _{n}},\\
\lim_{p\rightarrow-p_{0}}\left(  p+p_{0}\right)  \Gamma\left(  p_{0}%
+p-n\right)   &  =\lim_{p\rightarrow-p_{0}}\left(  p+p_{0}\right)
\frac{\Gamma\left(  p+p_{0}+1\right)  }{\left(  p+p_{0}-n\right)  _{n+1}}\\
&  =\frac{\Gamma\left(  1\right)  }{\left(  -n\right)  _{n}}=\frac{\left(
-1\right)  ^{n}}{n!}.
\end{align*}
Thus
\begin{equation}
f\left(  x\right)  =\left(  \prod_{i=1}^{m}\frac{\Gamma\left(  v_{i}\right)
}{\Gamma\left(  u_{i}\right)  }\right)  \sum_{k=1}^{m}\frac{x^{u_{k}-1}%
}{\Gamma\left(  v_{k}-u_{k}\right)  }\prod_{j\neq k}\frac{\Gamma\left(
u_{j}-u_{k}\right)  }{\Gamma\left(  v_{j}-u_{k}\right)  }\sum_{n=0}^{\infty
}\prod_{i=1}^{m}\frac{\left(  1+u_{k}-v_{i}\right)  _{n}}{\left(
1+u_{k}-u_{i}\right)  _{n}}x^{n},\label{sol0}%
\end{equation}
(note $\left(  1+u_{k}-u_{k}\right)  _{n}=n!);$in fact this is a Meijer
G-function (see \cite[16.17.2]{DLMF}).

\section{The differential equation}

The equation is of Mellin-Barnes type: let $\partial_{x}:=\frac{d}%
{dx},D:=x\partial_{x}$ and define the differential operator
\[
T\left(  u,v\right)  =-x\prod\limits_{j=1}^{m}\left(  D+2-v_{j}\right)
+\prod\limits_{j=1}^{m}\left(  D+1-u_{i}\right)  .
\]
The highest order term is $\left(  1-x\right)  x^{m}\partial_{x}^{m}$ and the
equation has regular singular points at $x=0$ and $x=1$. We find%
\begin{align*}
T\left(  u,v\right)  x^{c}\sum_{n=0}^{\infty}c_{n}x^{n} &  =x^{c}\sum
_{n=0}^{\infty}c_{n}\left\{  -\prod\limits_{j=1}^{m}\left(  n+c+2-v_{j}%
\right)  x^{n+1}+\prod\limits_{j=1}^{m}\left(  n+c+1-u_{i}\right)
x^{n}\right\}  \\
&  =x^{c}\sum_{n=1}^{\infty}x^{n}\left\{  c_{n}\prod\limits_{j=1}^{m}\left(
n+c+1-u_{i}\right)  -c_{n-1}\prod\limits_{j=1}^{m}\left(  n+c+1-v_{j}\right)
\right\}  \\
&  +x^{c}c_{0}\prod\limits_{j=1}^{m}\left(  c+1-u_{i}\right)  .
\end{align*}
The solutions of the indicial equation are $c=u_{i}-1,1\leq i\leq m$. Assume
$u_{i}-u_{j}\notin\mathbb{Z}$ for $i\neq j$. Let $c=u_{1}-1$ then obtain a
solution of $T\left(  u,v\right)  f\left(  x\right)  =0$ by solving the
recurrence
\[
c_{n}=\prod\limits_{j=1}^{m}\frac{\left(  u_{1}-v_{j}+n\right)  }{\left(
u_{1}-u_{j}+n\right)  }c_{n-1}=\frac{%
%TCIMACRO{\tprod \nolimits_{j=1}^{m}}%
%BeginExpansion
{\textstyle\prod\nolimits_{j=1}^{m}}
%EndExpansion
\left(  u_{1}-v_{j}+1\right)  _{n}}{n!%
%TCIMACRO{\tprod \nolimits_{j=2}^{m}}%
%BeginExpansion
{\textstyle\prod\nolimits_{j=2}^{m}}
%EndExpansion
\left(  u_{1}-u_{j}+1\right)  _{n}}c_{0}.
\]
Thus the solutions of $T\left(  u,v\right)  f=0$ are linear combinations of%
\begin{align*}
f_{1}\left(  x\right)   &  :=x^{u_{1}-1}~_{m}F_{m-1}\left(
%TCIMACRO{\QATOP{u_{1}-v_{1}+1,\ldots,u_{1}-v_{m}+1}{u_{1}-u_{2}+1,\ldots
%,u_{1}-u_{m}+1}}%
%BeginExpansion
\genfrac{}{}{0pt}{}{u_{1}-v_{1}+1,\ldots,u_{1}-v_{m}+1}{u_{1}-u_{2}%
+1,\ldots,u_{1}-u_{m}+1}%
%EndExpansion
;x\right)  ,\\
f_{i}\left(  x\right)   &  :=x^{u_{i}-1}\sum_{n=0}^{\infty}\prod_{j=1}%
^{m}\frac{\left(  u_{i}-v_{j}+1\right)  _{n}}{\left(  u_{i}-u_{j}+1\right)
_{n}}x^{n},
\end{align*}
for $1\leq i\leq m$ (note the factor $\left(  u_{i}-u_{i}+1\right)  _{n}=n!$).

\begin{lemma}
Suppose $g$ is differentiable on $(0,1]$, $g^{\left(  j\right)  }\left(
1\right)  =0$ for $0\leq j\leq k$ and $h\left(  x\right)  :=\left(
D+s\right)  g\left(  x\right)  $ then $h^{\left(  j\right)  }\left(  1\right)
=0$ for $0\leq j\leq k-1$. Furthermore if $k\geq0$ then for $n\geq0$
\[
\int_{0}^{1}x^{n}h\left(  x\right)  dx=\left(  s-n-1\right)  \int_{0}^{1}%
x^{n}g\left(  x\right)  dx.
\]

\end{lemma}

\begin{proof}
By induction $\partial_{x}^{j}D=x\partial_{x}^{j+1}+j\partial_{x}^{j}$ for
$j\geq0$. Hence $h^{\left(  j\right)  }\left(  1\right)  =g^{\left(
j+1\right)  }\left(  1\right)  +\left(  j+s\right)  g^{\left(  j\right)
}\left(  1\right)  $. Next%
\begin{align*}
\int_{0}^{1}x^{n}h\left(  x\right)  dx  &  =s\int_{0}^{1}x^{n}g\left(
x\right)  dx+\int_{0}^{1}x^{n+1}g^{\prime}\left(  x\right)  dx\\
&  =s\int_{0}^{1}x^{n}g\left(  x\right)  dx+g\left(  1\right)  -\left(
n+1\right)  \int_{0}^{1}x^{n}g\left(  x\right)  dx,
\end{align*}
and $g\left(  1\right)  =0$ by hypothesis.
\end{proof}

This is the fundamental initial value system:%
\begin{align}
T\left(  u,v\right)  f\left(  x\right)   &  =0,\label{Fsys}\\
f^{\left(  j\right)  }\left(  1\right)   &  =0,0\leq j\leq m-1\nonumber
\end{align}

\begin{proposition}
\label{diffeqmts}Suppose $f$ is a solution defined on $(0,1]$ of (\ref{Fsys})
then for $n\geq0$%
\[
\int_{0}^{1}x^{n}f\left(  x\right)  dx=\prod\limits_{i=1}^{m}\frac{\left(
u_{i}\right)  _{n}}{\left(  v_{i}\right)  _{n}}\int_{0}^{1}f\left(  x\right)
dx.
\]

\end{proposition}

\begin{proof}
For $0\leq j\leq m$ let $h_{j}=\prod_{i=1}^{j}\left(  D+1-u_{i}\right)  f$,
thus $h_{j+1}=\left(  D+1-u_{j+1}\right)  h_{j}$ and by the Lemma
$h_{j}^{\left(  k\right)  }\left(  1\right)  =0$ for $0\leq k\leq m-1-j$. Also
$\int_{0}^{1}x^{n}h_{j+1}\left(  x\right)  dx=-\left(  n+u_{j+1}\right)
\int_{0}^{1}x^{n}h_{j}\left(  x\right)  dx$ for $0\leq j\leq m-1$. By
induction $\int_{0}^{1}x^{n}h_{m}\left(  x\right)  dx=\left(  -1\right)
^{m}\prod_{i=1}^{m}\left(  n+u_{i}\right)  \int_{0}^{1}x^{n}f\left(  x\right)
dx$.

Similarly $\int_{0}^{1}x^{n}x\prod_{i=1}^{j}\left(  D+2-v_{i}\right)  f\left(
x\right)  dx=\left(  -1\right)  ^{m}\prod_{i=1}^{m}\left(  n+v_{i}\right)
\int_{0}^{1}x^{n+1}f\left(  x\right)  dx$. Thus the integral $0=\int_{0}%
^{1}x^{n}T\left(  u,v\right)  f\left(  x\right)  dx$ implies the recurrence%
\[
\int_{0}^{1}x^{n+1}f\left(  x\right)  dx=\prod_{i=1}^{m}\frac{\left(
u_{i}+n\right)  }{\left(  v_{i}+n\right)  }\int_{0}^{1}x^{n}f\left(  x\right)
dx.
\]
Induction completes the proof.
\end{proof}

Observe that the coefficients $\left\{  \gamma_{i}\right\}  $ of the solution
$\sum_{i=1}^{m}\gamma_{i}f_{i}\left(  x\right)  $ of the system are not
explicit here, but they are found in the inverse Mellin transform expression.

\section{The behavior near $x=1$ and the recurrence}

First we establish the form of the density $f\left(  x\right)  $ in terms of
powers of $\left(  1-x\right)  $.

\begin{lemma}
\label{inttxt}For $\alpha,\beta,\gamma>0$ and $0<x\leq1$%
\[
\int_{x}^{1}t^{\alpha-1}\left(  1-t\right)  ^{\beta-1}\left(  1-\frac{x}%
{t}\right)  ^{\gamma-1}dt=B\left(  \beta,\gamma\right)  \left(  1-x\right)
^{\beta+\gamma-1}~_{2}F_{1}\left(
%TCIMACRO{\QATOP{\gamma-\alpha,\beta}{\beta+\gamma}}%
%BeginExpansion
\genfrac{}{}{0pt}{}{\gamma-\alpha,\beta}{\beta+\gamma}%
%EndExpansion
;1-x\right)  .
\]

\end{lemma}

\begin{proof}
Change the variable of integration $t=1-s+sx$ then the integral becomes%
\begin{align*}
&  \left(  1-x\right)  ^{\beta+\gamma-1}\int_{0}^{1}\left(  1-s\left(
1-x\right)  \right)  ^{\alpha-\gamma}s^{\beta-1}\left(  1-s\right)
^{\gamma-1}ds\\
&  =\left(  1-x\right)  ^{\beta+\gamma-1}\frac{\Gamma\left(  \beta\right)
\Gamma\left(  \gamma\right)  }{\Gamma\left(  \beta+\gamma\right)  }~_{2}%
F_{1}\left(
%TCIMACRO{\QATOP{\gamma-\alpha,\beta}{\beta+\gamma}}%
%BeginExpansion
\genfrac{}{}{0pt}{}{\gamma-\alpha,\beta}{\beta+\gamma}%
%EndExpansion
;1-x\right)  .
\end{align*}
This is a standard formula, see \cite[(9.1.4), p.239]{LE} and is valid in
$0<x\leq1$ (where $\left\vert 1-x\right\vert <1$).
\end{proof}

Set $\delta:=\sum_{i=1}^{m}\left(  v_{i}-u_{i}\right)  $.

\begin{proposition}
\label{series1-x}There exists a sequence $\left\{  c_{n}\right\}  $ such that
\[
f\left(  x\right)  =\frac{1}{\Gamma\left(  \delta\right)  }\prod_{i=1}%
^{m}\frac{\Gamma\left(  v_{i}\right)  }{\Gamma\left(  u_{i}\right)  }\left(
1-x\right)  ^{\delta-1}\left\{  1+\sum_{n=1}^{\infty}c_{n}\left(  1-x\right)
^{n}\right\}  .
\]

\end{proposition}

\begin{proof}
Argue by induction. For $m=1$ we have (see (\ref{fbeta}))%
\begin{align*}
f\left(  x\right)   &  =\frac{\Gamma\left(  v_{1}\right)  }{\Gamma\left(
u_{1}\right)  \Gamma\left(  v_{1}-u_{1}\right)  }x^{u_{1}-1}\left(
1-x\right)  ^{v_{1}-u_{1}-1}\\
&  =\frac{\Gamma\left(  v_{1}\right)  }{\Gamma\left(  u_{1}\right)
\Gamma\left(  v_{1}-u_{1}\right)  }\left(  1-x\right)  ^{v_{1}-u_{1}%
-1}\left\{  1+\sum_{n=1}^{\infty}\frac{\left(  1-u_{1}\right)  _{n}}%
{n!}\left(  1-x\right)  ^{n}\right\}  .
\end{align*}
Assume the statement is proven for some $m\geq1$, then $g=f\ast h\left(
u_{m+1},v_{m+1}-u_{m+1};\cdot\right)  $ has the moments $\prod\limits_{i=1}%
^{m+1}\frac{\left(  u_{i}\right)  _{n}}{\left(  v_{i}\right)  _{n}}$. The
convolution integral (see (\ref{convol})) is a sum of terms%
\begin{align*}
&  C_{n}\int_{x}^{1}\left(  1-t\right)  ^{\delta+n-1}\left(  \frac{x}%
{t}\right)  ^{u_{m+1}-1}\left(  1-\frac{x}{t}\right)  ^{v_{m+1}-u_{m+1}%
-1}\frac{dt}{t}\\
&  =C_{n}x^{u_{m+1}-1}\left(  1-x\right)  ^{\delta+n+v_{m+1}-u_{m+1}-1}%
\frac{\Gamma\left(  \delta+n\right)  \Gamma\left(  v_{m+1}-u_{m+1}\right)
}{\Gamma\left(  \delta+n+v_{m+1}-u_{m+1}\right)  }\\
&  \times~_{2}F_{1}\left(
%TCIMACRO{\QATOP{\delta+n,v_{m+1}-1}{\delta+v_{m+1}-u_{m+1}+n}}%
%BeginExpansion
\genfrac{}{}{0pt}{}{\delta+n,v_{m+1}-1}{\delta+v_{m+1}-u_{m+1}+n}%
%EndExpansion
;1-x\right)
\end{align*}
by Lemma \ref{inttxt} ; and $x^{u_{m+1}-1}=1+\sum_{j=1}^{\infty}\frac{\left(
1-u_{m+1}\right)  _{j}}{j!}\left(  1-x\right)  ^{j}$. Thus the lowest power of
$\left(  1-x\right)  $ appearing in $g$ is $\delta+v_{m+1}-u_{m+1}-1$ which
occurs for $n=0$. By the inductive hypothesis%
\[
C_{0}=\frac{1}{\Gamma\left(  \delta\right)  }\prod_{i=1}^{m}\frac
{\Gamma\left(  v_{i}\right)  }{\Gamma\left(  u_{i}\right)  }\frac
{\Gamma\left(  v_{m+1}\right)  }{\Gamma\left(  u_{m+1}\right)  \Gamma\left(
v_{m+1}-u_{m+1}\right)  },
\]
and so the coefficient of $\left(  1-x\right)  ^{\delta+v_{m+1}-u_{m+1}-1}$ in
$g$ is%
\[
C_{0}\frac{\Gamma\left(  \delta\right)  \Gamma\left(  v_{m+1}-u_{m+1}\right)
}{\Gamma\left(  \delta+v_{m+1}-u_{m+1}\right)  }=\frac{1}{\Gamma\left(
\delta+v_{m+1}-u_{m+1}\right)  }\prod_{i=1}^{m+1}\frac{\Gamma\left(
v_{i}\right)  }{\Gamma\left(  u_{i}\right)  };
\]
this completes the induction.
\end{proof}

For the next step we need to express $T\left(  u,v\right)  $ in the form
$x^{m}\left(  1-x\right)  \partial_{x}^{m}+\sum_{j=0}^{m-1}x^{j}\left(
a_{j}-b_{j}x\right)  \partial_{x}^{j}$. Recall the elementary symmetric
polynomials in the variables $\left\{  z_{1},\ldots,z_{m}\right\}  $ given by
the generating function%
\[
\prod_{j=1}^{m}\left(  q+z_{j}\right)  =\sum_{n=0}^{m}e_{n}\left(  z\right)
q^{m-n},
\]
so $e_{0}\left(  z\right)  =1,e_{1}\left(  z\right)  =z_{1}+z_{2}+\ldots
+z_{m},e_{2}\left(  z\right)  =z_{1}z_{2}+\ldots+z_{m-1}z_{m}$ and
$e_{m}\left(  z\right)  =z_{1}z_{2}\ldots z_{m}$. Thus%
\[
\prod\limits_{j=1}^{m}\left(  D+1-u_{i}\right)  =\sum_{j=0}^{m}e_{m-j}\left(
1-u\right)  \left(  x\partial_{x}\right)  ^{j}=\sum_{j=0}^{m}a_{j}%
x^{j}\partial_{x}^{j}.
\]
Let $\left(  x\partial_{x}\right)  ^{k}=\sum_{i=0}^{j}A_{k,i}x^{i}\partial
_{x}^{i}$ then $x\partial_{x}\left(  x\partial_{x}\right)  ^{j}=\sum_{i=0}%
^{j}A_{k,i}\left(  ix^{i}\partial_{x}^{i}+x^{i+1}\partial_{x}^{i+1}\right)  $,
so $A_{k+1,i}=A_{k,i-1}+iA_{k,i}$. This recurrence has the boundary values
$A_{0,0}=1,A_{1,0}=0,A_{1,1}=1$. The solution consists of the Stirling numbers
of the second kind, denoted $S\left(  k,i\right)  $ (see \cite[26.8.22]%
{DLMF}). Thus
\begin{align*}
\sum_{j=0}^{m}a_{j}x^{j}\partial_{x}^{j} &  =\sum_{j=0}^{m}e_{m-j}\left(
1-u\right)  \sum_{i=0}^{j}S\left(  j,i\right)  x^{i}\partial_{x}^{i}\\
&  =\sum_{j=0}^{m}x^{j}\partial_{x}^{j}\sum_{i=j}^{m}S\left(  i,j\right)
e_{m-i}\left(  1-u\right)  ,\\
a_{j} &  =\sum_{i=j}^{m}S\left(  i,j\right)  e_{m-i}\left(  1-u\right)  ,0\leq
j\leq m.
\end{align*}
In particular $a_{m}=1,$ $a_{m-1}=\binom{m}{2}+e_{1}\left(  1-u\right)  $, and
$a_{0}=e_{m}\left(  1-u\right)  =\prod_{j=1}^{m}\left(  1-u_{j}\right)  $.
Similarly%
\begin{align*}
\sum_{j=0}^{m}b_{j}x^{j}\partial_{x}^{j} &  =\prod\limits_{j=1}^{m}\left(
D+2-v_{j}\right)  =\sum_{j=0}^{m}e_{m-j}\left(  2-v\right)  \left(
x\partial_{x}\right)  ^{j}\\
&  =\sum_{j=0}^{m}x^{j}\partial_{x}^{j}\sum_{i=j}^{m}S\left(  i,j\right)
e_{m-j}\left(  2-v\right)  ,\\
b_{j} &  =\sum_{i=j}^{m}S\left(  i,j\right)  e_{m-j}\left(  2-v\right)  ,0\leq
j\leq m.
\end{align*}

The differential equation leads to deriving recurrence relations for the
coefficients $\left\{  c_{n}\right\}  $. Convert the differential operator
$T\left(  u,v\right)  $ to the coordinate $t=1-x$; set $\partial_{t}:=\frac
{d}{dt}$ (so that $\partial_{t}=-\partial_{x}$). Write (expanding
$x^{j}=\left(  1-t\right)  ^{j}$ with the binomial theorem)%
\begin{align*}
T\left(  u,v\right)   &  =-x\sum_{j=0}^{m}b_{j}x^{j}\partial_{x}^{j}%
+\sum_{j=0}^{m}a_{j}x^{j}\partial_{x}^{j}\\
&  =-\sum_{j=0}^{m}\sum_{i=0}^{j+1}\binom{j+1}{i}\left(  -t\right)  ^{i}%
b_{j}\left(  -\partial_{t}\right)  ^{j}+\sum_{j=0}^{m}\sum_{i=0}^{j}\binom
{j}{i}\left(  -t\right)  ^{i}a_{j}\left(  -\partial_{t}\right)  ^{j}\\
&  =\sum_{k=-m}^{1}\left(  -1\right)  ^{k}\sum_{j=\max\left(  0,-k\right)
}^{m}\left\{  \binom{j}{j+k}a_{j}-\binom{j+1}{j+k}b_{j}\right\}
t^{k+j}\partial_{t}^{j}.
\end{align*}
The highest order term ($k=1$) is $\sum_{j=0}^{m}b_{j}t^{j+1}\partial_{t}^{j}%
$. The term with $k=0$ is $\sum_{j=0}^{m}\left(  a_{j}-\left(  j+1\right)
b_{j}\right)  t^{j}\partial_{t}^{j}$. The two bottom terms ($k=-m,1-m$) are%
\begin{align*}
\left(  -1\right)  ^{m}\left(  a_{m}-b_{m}\right)  \partial_{t}^{m}  &  =0,\\
\left(  -1\right)  ^{m-1}\left(  a_{m-1}-b_{m-1}-t\partial_{t}\right)
\partial_{t}^{m-1}  &  =\left(  -1\right)  ^{m-1}\left(  \sum_{i=1}^{m}\left(
v_{i}-u_{i}\right)  -m-t\partial_{t}\right)  \partial_{t}^{m-1};
\end{align*}
the remaining terms ($-m<-k<0$) are
\[
\left(  -1\right)  ^{k}\left(  \sum_{j=k}^{m}\left(  \binom{j}{j-k}%
a_{j}-\binom{j+1}{j-k}b_{j}\right)  t^{j-k}\partial_{t}^{j-k}\right)
\partial_{t}^{k}.
\]
Apply $T\left(  u,v\right)  $ to $t^{\gamma}$ (with the aim of finding a
solution $t^{c}\sum_{n=0}^{\infty}c_{n}t^{n}$ to $T\left(  u,v\right)  f=0$);
note $\partial_{t}^{j}t^{\gamma}=\left(  -1\right)  ^{j}\left(  -\gamma
\right)  _{j}t^{\gamma-j}$, then%
\begin{align*}
T\left(  u,v\right)  t^{\gamma}  &  =\sum_{k=-1}^{m-1}R_{k}\left(
\gamma\right)  t^{\gamma-k},\\
R_{-1}\left(  \gamma\right)   &  =\sum_{j=0}^{m}b_{j}\left(  -1\right)
^{j}\left(  -\gamma\right)  _{j},\\
R_{m-1}\left(  \gamma\right)   &  =\left(  -\gamma\right)  _{m-1}\left(
\sum_{i=1}^{m}\left(  v_{i}-u_{i}\right)  -\gamma-1\right)  =\left(
-\gamma\right)  _{m-1}\left(  \delta-1-\gamma\right)  ,
\end{align*}
and
\begin{align*}
R_{k}\left(  \gamma\right)   &  =\left(  -\gamma\right)  _{k}R_{k}^{\prime
}\left(  \gamma\right)  ,\\
R_{k}^{\prime}  &  =\left(  \sum_{j=k}^{m}\left(  \binom{j}{j-k}a_{j}%
-\binom{j+1}{j-k}b_{j}\right)  \left(  -1\right)  ^{j-k}\left(  -\gamma
+k\right)  _{j-k}\right)  ,0\leq k<m.
\end{align*}
These sums can be considerably simplified (and the Stirling numbers are not
needed). Introduce the difference operator
\[
\nabla g\left(  c\right)  =g\left(  c\right)  -g\left(  c-1\right)  .
\]
This has a convenient action, for $j\geq0$ and arbitrary $k$
\begin{align*}
\nabla\left(  k-c\right)  _{j}  &  =\left(  k-c\right)  _{j}-\left(
k+1-c\right)  _{j}=\left\{  k-c-\left(  k+j-c\right)  \right\}  \left(
k+1-c\right)  _{j-1}\\
&  =-j\left(  k+1-c\right)  _{j-1},\\
\nabla^{k}\left(  -c\right)  _{j}  &  =\left(  -1\right)  ^{k}\frac
{j!}{\left(  j-k\right)  !}\left(  -c+k\right)  _{j-k},k\leq j.
\end{align*}
Define the polynomials
\begin{align*}
p\left(  c\right)   &  =\prod_{i=1}^{m}\left(  c+1-u_{i}\right)  ,q\left(
c\right)  =\prod_{i=1}^{m}\left(  c+2-v_{i}\right) \\
q_{1}\left(  c\right)   &  =\left(  1+c\nabla\right)  q\left(  c\right)
=\left(  1+c\right)  q\left(  c\right)  -cq\left(  c-1\right)  .
\end{align*}

\begin{proposition}
$R_{-1}\left(  \gamma\right)  =q\left(  \gamma\right)  $ and for $0\leq k\leq
m-1$%
\[
R_{k}^{\prime}\left(  \gamma\right)  =\frac{1}{k!}\nabla^{k}p\left(
\gamma\right)  -\frac{1}{\left(  k+1\right)  !}\nabla^{k}q_{1}\left(
\gamma\right)  .
\]

\end{proposition}

\begin{proof}
By construction $\sum_{j=0}^{m}a_{j}t^{j}\partial_{t}^{j}t^{\gamma}=p\left(
\gamma\right)  $ and $\sum_{j=0}^{m}b_{j}t^{j}\partial_{t}^{j}t^{\gamma
}=q\left(  \gamma\right)  $. Apply $\frac{1}{k!}\nabla^{k}$ to both sides
($\nabla$ acts on the variable $\gamma$) of%
\[
p\left(  \gamma\right)  =\sum_{j=0}^{m}a_{j}t^{j}\partial_{t}^{j}t^{\gamma
}=\sum_{j=0}^{m}\left(  -1\right)  ^{j}a_{j}\left(  -\gamma\right)  _{j},
\]
to obtain%
\begin{align*}
\frac{1}{k!}\nabla^{k}p\left(  \gamma\right)   &  =\frac{1}{k!}\sum_{j=k}%
^{m}\left(  -1\right)  ^{j-k}\frac{j!}{\left(  j-k\right)  !}a_{j}\left(
-\gamma+k\right)  _{j-k}\\
&  =\sum_{j=k}^{m}\left(  -1\right)  ^{j-k}\binom{j}{j-k}a_{j}\left(
-\gamma+k\right)  _{j-k}.
\end{align*}
Also%
\begin{align*}
q_{1}\left(  \gamma\right)   &  =\left(  1+\gamma\nabla\right)  q\left(
\gamma\right)  =\left(  1+\gamma\nabla\right)  \sum_{j=0}^{m}b_{j}%
t^{j}\partial_{t}^{j}t^{\gamma}=\left(  1+\gamma\nabla\right)  \sum_{j=0}%
^{m}\left(  -1\right)  ^{j}b_{j}\left(  -\gamma\right)  _{j}\\
&  =\sum_{j=0}^{m}\left(  -1\right)  ^{j}b_{j}\left\{  \left(  -\gamma\right)
_{j}-j\gamma\left(  1-\gamma\right)  _{j-1}\right\}  =\sum_{j=0}^{m}\left(
-1\right)  ^{j}b_{j}\left(  1+j\right)  \left(  -\gamma\right)  _{j}.
\end{align*}
Apply $\frac{1}{\left(  k+1\right)  !}\nabla^{k}$ to both sides to obtain
\begin{align*}
\frac{1}{\left(  k+1\right)  !}\nabla^{k}q_{1}\left(  \gamma\right)   &
=\frac{1}{\left(  k+1\right)  !}\sum_{j=k}^{m}\left(  -1\right)  ^{j-k}%
b_{j}\left(  1+j\right)  \frac{j!}{\left(  j-k\right)  !}\left(
-\gamma+k\right)  _{j-k}\\
&  =\sum_{j=k}^{m}\left(  -1\right)  ^{j-k}\binom{j+1}{j-k}b_{j}\left(
-\gamma+k\right)  _{j-k}.
\end{align*}
This completes the proof.
\end{proof}

Hence%
\begin{align*}
T\left(  u,v\right)  t^{c}\sum_{n=0}^{\infty}c_{n}t^{n}  &  =t^{c}\sum
_{n=0}^{\infty}c_{n}\sum_{k=-1}^{m-1}R_{k}\left(  n+c\right)  t^{n-k}\\
&  =t^{c}\sum_{n=1-m}^{\infty}t^{n}\sum_{k=-1}^{m-1}R_{k}\left(  n+k+c\right)
c_{n+k}.
\end{align*}
The recurrence for the coefficients for $n\geq m$ is%
\begin{gather*}
R_{m-1}\left(  n+c\right)  c_{n}=-\sum_{k=1}^{\min\left(  m,n\right)
}R_{m-1-k}\left(  n-k+c\right)  c_{n-k},\\
\left(  -n-c\right)  _{m-1}\left(  \delta-1-n-c\right)  c_{n}=\\
-\sum_{k=1}^{\min\left(  m-1,n\right)  }\left(  -n+k-c\right)  _{m-1-k}%
R_{m-1-k}^{\prime}\left(  n-k+c\right)  c_{n-k}-q\left(  n-m+c\right)
c_{n-m},
\end{gather*}
where $c_{i}=0$ for $i<0$. At $n=0$ the equation is $\left(  -c\right)
_{m-1}\left(  \delta-1-c\right)  c_{0}=0$. Let $c=0$ then for $0\leq n\leq
m-2$ the equations are%
\[
\left(  -n\right)  _{m-1}\left(  \delta-1-n\right)  c_{n}=-\sum_{k=1}%
^{n}\left(  -n+k\right)  _{m-1-k}R_{m-1-k}^{\prime}\left(  n-k\right)
c_{n-k},
\]
but $\left(  -n+k\right)  _{m-1-k}=0$ for $m-1-k>n-k$ (and $0\leq k\leq n$)
thus the coefficients $c_{0},c_{1},\ldots,c_{m-2}$ are arbitrary, providing
$m-1$ linearly independent solutions to $T\left(  u,v\right)  f=0$. The
recurrence can be rewritten as%
\begin{align*}
c_{m-1}  &  =\frac{-1}{\left(  m-1\right)  !\left(  \delta-m\right)  }%
\sum_{k=1}^{m-1}\left(  -1\right)  ^{k}\left(  m-1-k\right)  !R_{m-1-k}%
^{\prime}\left(  n-k\right)  c_{n-k},\\
c_{n}  &  =\frac{-1}{\left(  \delta-1-n\right)  }\left\{  \sum_{k=1}%
^{m-1}\frac{1}{\left(  -n\right)  _{k}}R_{m-1-k}^{\prime}\left(  n-k\right)
c_{n-k}+\frac{1}{\left(  -n\right)  _{m-1}}q\left(  n-m\right)  c_{n-m}%
\right\}  ,n\geq m.
\end{align*}

Assume that $\delta\notin\mathbb{Z}$ to avoid poles. But these are different
from the desired solution which has $c=\delta-1$ as was shown in Proposition
\ref{series1-x}. The recurrence behaves better in this case. Indeed
\begin{align*}
c_{n}  &  =\frac{1}{n}\sum_{k=1}^{\min\left(  m-1,n\right)  }\frac{\left(
-n+k-\delta+1\right)  _{m-1-k}}{\left(  -n-\delta+1\right)  _{m-1}}%
R_{m-1-k}^{\prime}\left(  n-k+\delta-1\right)  c_{n-k}\\
&  +\frac{1}{n\left(  -n-\delta+1\right)  _{m-1}}q\left(  n-m+\delta-1\right)
c_{n-m},
\end{align*}
which simplifies to
\begin{align}
c_{n}  &  =\frac{1}{n}\sum_{k=1}^{\min\left(  m-1,n\right)  }\frac{1}{\left(
-n-\delta+1\right)  _{k}}R_{m-1-k}^{\prime}\left(  n-k+\delta-1\right)
c_{n-k}\label{recF}\\
&  +\frac{1}{n\left(  -n-\delta+1\right)  _{m-1}}q\left(  n-m+\delta-1\right)
c_{n-m}.\nonumber
\end{align}
The term with $c_{n-m}$ occurs only for $n\geq m$. The denominator factors are
of the form $\left(  \delta+n-1\right)  \left(  \delta+n-2\right)
\ldots\left(  \delta+n-k\right)  $. If $n\geq m$ then the smallest factor is
$\delta+n-m>\delta>0$; otherwise the smallest factor is $\delta$ (for $k=n$).
Hence this solution is well-defined for any $\delta>0$.

\begin{theorem}
\label{thm1mx}Suppose $u_{1},\ldots,u_{m},v_{1},\ldots,v_{m}$ satisfy
$v_{i}>u_{i}>0$ for each $i$ then there is a density function $f$ on $\left[
0,1\right]  $ with moment sequence $\prod\limits_{i=1}^{m}\frac{\left(
u_{i}\right)  _{n}}{\left(  v_{i}\right)  _{n}}$ and%
\[
f\left(  x\right)  =\frac{1}{\Gamma\left(  \delta\right)  }\prod_{i=1}%
^{m}\frac{\Gamma\left(  v_{i}\right)  }{\Gamma\left(  u_{i}\right)  }\left(
1-x\right)  ^{\delta-1}\left\{  1+\sum_{n=1}^{\infty}c_{n}\left(  1-x\right)
^{n}\right\}  ,
\]
where the coefficients $\left\{  c_{n}\right\}  $ are obtained with the
recurrence (\ref{recF}) using $c_{0}=1$, and $\delta=\sum_{i=1}^{m}\left(
v_{i}-u_{i}\right)  $.
\end{theorem}

\begin{proof}
The density exists because it is the distribution of the random variable
$\prod_{i=1}^{m}X_{i}$ where the $X_{i}$'s are jointly independent and the
moments of $X_{i}$ are $\frac{\left(  u_{i}\right)  _{n}}{\left(
v_{i}\right)  _{n}}$ for each $i$. By Proposition \ref{series1-x} $f$ has the
series expansion given in the statement. Let $g\left(  x\right)  $ be the
function given in the statement and suppose for now that $\delta>m$ then $g$
is a solution of the differential system (\ref{Fsys}) (because of the factor
$\left(  1-x\right)  ^{\delta-1}$). By Proposition \ref{diffeqmts} $g$ has the
same moments as $Cf$ for some constant $C$. By Proposition \ref{series1-x} $f$
and $g$ have the same leading coefficient in their series expansions. Hence
$f=g$. The coefficients $c_{n}$ are analytic in the parameters for the range
$\delta>m$. Each moment $\int_{0}^{1}x^{n}f\left(  x\right)  dx$ is similarly
analytic and so the formula is valid for all $\delta>0$, by use of analytic
continuation from the range $\delta>m$.
\end{proof}

The coefficients occurring in the recurrence (\ref{recF}) are expressions in
the parameters $u,v$, which can be straightforwardly computed, especially with
computer symbolic algebra.

\section{Examples}

\subsection{Density for $m=2$}

For the easy case $m=2$ we can directly find the density function, in a
slightly different form.

Given $u_{1},u_{2},v_{1},v_{2}>0$ and $\delta=v_{1}+v_{2}-u_{1}-u_{2}>0$ set%
\begin{align*}
g\left(  u,v;x\right)   &  =\frac{\Gamma\left(  v_{1}\right)  \Gamma\left(
v_{2}\right)  }{\Gamma\left(  u_{1}\right)  \Gamma\left(  u_{2}\right)
\Gamma\left(  \delta\right)  }\\
&  \times x^{u_{2}-1}\left(  1-x\right)  ^{\delta-1}~_{2}F_{1}\left(
%TCIMACRO{\QATOP{v_{2}-u_{2},v_{1}-u_{2}}{\delta}}%
%BeginExpansion
\genfrac{}{}{0pt}{}{v_{2}-u_{2},v_{1}-u_{2}}{\delta}%
%EndExpansion
;1-x\right)
\end{align*}
then%
\[
\int_{0}^{1}x^{n}g\left(  u,v;x\right)  dx=\frac{\left(  u_{1}\right)
_{n}\left(  u_{2}\right)  _{n}}{\left(  v_{1}\right)  _{n}\left(
v_{2}\right)  _{n}},n=0,1,2,\ldots.
\]

\begin{proof}
Consider%
\begin{align*}
&  \int_{0}^{1}x^{n+u_{1}-1}\left(  1-x\right)  ^{\delta-1}~_{2}F_{1}\left(
%TCIMACRO{\QATOP{v_{2}-u_{2},v_{1}-u_{2}}{\delta}}%
%BeginExpansion
\genfrac{}{}{0pt}{}{v_{2}-u_{2},v_{1}-u_{2}}{\delta}%
%EndExpansion
;1-x\right)  dx\\
&  =\sum_{m=0}^{\infty}\frac{\left(  v_{2}-u_{2}\right)  _{m}\left(
v_{1}-u_{2}\right)  _{m}}{m!\left(  \delta\right)  _{m}}\frac{\Gamma\left(
n+u_{1}\right)  \Gamma\left(  \delta+m\right)  }{\Gamma\left(  u_{1}%
+\delta+m+n\right)  }\\
&  =\frac{\Gamma\left(  u_{1}\right)  \left(  u_{1}\right)  _{n}\Gamma\left(
\delta\right)  }{\Gamma\left(  n+u_{1}+\delta\right)  }\sum_{m=0}^{\infty
}\frac{\left(  v_{2}-u_{2}\right)  _{m}\left(  v_{1}-u_{2}\right)  _{m}%
}{m!\left(  n+u_{1}+\delta\right)  _{m}}\\
&  =\frac{\Gamma\left(  u_{1}\right)  \left(  u_{1}\right)  _{n}\Gamma\left(
\delta\right)  }{\Gamma\left(  n+u_{1}+\delta\right)  }\frac{\Gamma\left(
n+u_{1}+\delta\right)  \Gamma\left(  n+u_{1}+\delta-v_{2}-v_{1}+2u_{2}\right)
}{\Gamma\left(  n+u_{1}+\delta-v_{2}+u_{2}\right)  \Gamma\left(
n+u_{1}+\delta-v_{1}+u_{2}\right)  }\\
&  =\frac{\Gamma\left(  u_{1}\right)  \Gamma\left(  u_{2}\right)
\Gamma\left(  v_{2}+v_{1}-u_{1}-u_{2}\right)  }{\Gamma\left(  v_{2}\right)
\Gamma\left(  v_{1}\right)  }\frac{\left(  u_{1}\right)  _{n}\left(
u_{2}\right)  _{n}}{\left(  v_{2}\right)  _{n}\left(  v_{1}\right)  _{n}},
\end{align*}
for each $n$.
\end{proof}

By using standard transformations we can explain the other formulation for
$g\left(  u,v;x\right)  $ near $x=0.$ From \cite[p.249, (9.5.7)]{LE}%
\begin{align*}
F\left(
%TCIMACRO{\QATOP{\alpha,\delta}{\gamma}}%
%BeginExpansion
\genfrac{}{}{0pt}{}{\alpha,\delta}{\gamma}%
%EndExpansion
;1-x\right)   &  =\frac{\Gamma\left(  \gamma-\alpha-\delta\right)
\Gamma\left(  \gamma\right)  }{\Gamma\left(  \gamma-\alpha\right)
\Gamma\left(  \gamma-\delta\right)  }F\left(
%TCIMACRO{\QATOP{\alpha,\delta}{1+\alpha+\delta-\gamma}}%
%BeginExpansion
\genfrac{}{}{0pt}{}{\alpha,\delta}{1+\alpha+\delta-\gamma}%
%EndExpansion
;x\right)  \\
&  +\frac{\Gamma\left(  \alpha+\delta-\gamma\right)  \Gamma\left(
\gamma\right)  }{\Gamma\left(  \alpha\right)  \Gamma\left(  \delta\right)
}x^{\gamma-\alpha-\delta}F\left(
%TCIMACRO{\QATOP{\gamma-\alpha,\gamma-\delta}{1+\gamma-\alpha-\delta}}%
%BeginExpansion
\genfrac{}{}{0pt}{}{\gamma-\alpha,\gamma-\delta}{1+\gamma-\alpha-\delta}%
%EndExpansion
;x\right)  .
\end{align*}
applied to $g\left(  u,v;x\right)  $ (provided $u_{1}-u_{2}\notin\mathbb{Z}$)
we find%
\begin{align}
g\left(  u,v;x\right)   &  =\frac{\Gamma\left(  v_{2}\right)  \Gamma\left(
v_{1}\right)  \Gamma\left(  u_{2}-u_{1}\right)  }{\Gamma\left(  u_{1}\right)
\Gamma\left(  u_{2}\right)  \Gamma\left(  v_{2}-u_{1}\right)  \Gamma\left(
v_{1}-u_{1}\right)  }\label{2F1X}\\
&  \times x^{u_{1}-1}\left(  1-x\right)  ^{\delta-1}~_{2}F_{1}\left(
%TCIMACRO{\QATOP{v_{2}-u_{2},v_{1}-u_{2}}{1+u_{1}-u_{2}}}%
%BeginExpansion
\genfrac{}{}{0pt}{}{v_{2}-u_{2},v_{1}-u_{2}}{1+u_{1}-u_{2}}%
%EndExpansion
;x\right)  \nonumber\\
&  +\frac{\Gamma\left(  v_{2}\right)  \Gamma\left(  v_{1}\right)
\Gamma\left(  u_{1}-u_{2}\right)  }{\Gamma\left(  u_{1}\right)  \Gamma\left(
u_{2}\right)  \Gamma\left(  v_{2}-u_{2}\right)  \Gamma\left(  v_{1}%
-u_{2}\right)  }\nonumber\\
&  \times x^{u_{2}-1}\left(  1-x\right)  ^{\delta-1}~_{2}F_{1}\left(
%TCIMACRO{\QATOP{v_{2}-u_{1},v_{1}-u_{1}}{1+u_{2}-u_{1}}}%
%BeginExpansion
\genfrac{}{}{0pt}{}{v_{2}-u_{1},v_{1}-u_{1}}{1+u_{2}-u_{1}}%
%EndExpansion
;x\right)  .\nonumber
\end{align}
This is quite similar to the general formula (\ref{sol0}), and the following
standard transformation explains the difference%
\begin{equation}
_{2}F_{1}\left(
%TCIMACRO{\QATOP{a,b}{c}}%
%BeginExpansion
\genfrac{}{}{0pt}{}{a,b}{c}%
%EndExpansion
;x\right)  =\left(  1-x\right)  ^{c-a-b}~_{2}F_{1}\left(
%TCIMACRO{\QATOP{c-a,c-b}{c}}%
%BeginExpansion
\genfrac{}{}{0pt}{}{c-a,c-b}{c}%
%EndExpansion
;x\right)  .\label{2F1trans}%
\end{equation}
If $u_{1}-u_{2}\in\mathbb{Z}$ then there are terms in $\log x$. The relevant
formula can be found in \cite[p. 257, (9.7.5)]{LE}. Suppose $u_{2}=u_{1}+n$
and $n=0,1,2,\ldots,$ $\delta=v_{2}+v_{1}-2u_{1}-n$ then%
\begin{gather*}
g\left(  u,v;x\right)  =\frac{\Gamma\left(  v_{2}\right)  \Gamma\left(
v_{1}\right)  }{\Gamma\left(  u_{1}\right)  \Gamma\left(  u_{1}+n\right)
\Gamma\left(  v_{2}-u_{1}\right)  \Gamma\left(  v_{1}-u_{1}\right)  }\left(
1-x\right)  ^{\delta-1}\\
\times x^{u_{1}-1}\sum_{k=0}^{n-1}\frac{\left(  n-k-1\right)  !}{k!}\left(
v_{2}-u_{1}-n\right)  _{k}\left(  v_{1}-u_{1}-n\right)  _{k}\left(  -x\right)
^{k}\\
+\frac{\Gamma\left(  v_{2}\right)  \Gamma\left(  v_{1}\right)  }{\Gamma\left(
u_{1}\right)  \Gamma\left(  u_{1}+n\right)  \Gamma\left(  v_{2}-u_{1}%
-n\right)  \Gamma\left(  v_{1}-u_{1}-n\right)  }\left(  1-x\right)
^{\delta-1}\\
\times\left(  -1\right)  ^{n}x^{u_{1}+n-1}\left(  -\log x\right)  \frac{1}%
{n!}~_{2}F_{1}\left(
%TCIMACRO{\QATOP{v_{2}-u_{1},v_{1}-u_{1}}{n+1}}%
%BeginExpansion
\genfrac{}{}{0pt}{}{v_{2}-u_{1},v_{1}-u_{1}}{n+1}%
%EndExpansion
;x\right)  \\
+\frac{\left(  -1\right)  ^{n}\Gamma\left(  v_{2}\right)  \Gamma\left(
v_{1}\right)  }{\Gamma\left(  u_{1}\right)  \Gamma\left(  u_{1}+n\right)
\Gamma\left(  v_{2}-u_{1}-n\right)  \Gamma\left(  v_{1}-u_{1}-n\right)
}x^{u_{1}+n-1}\left(  1-x\right)  ^{\delta-1}\\
\times\sum_{k=0}^{\infty}\frac{\left(  v_{2}-u_{1}\right)  _{k}\left(
v_{1}-u_{1}\right)  _{k}}{k!\left(  n+k\right)  !}\\
\left\{  \psi\left(  k+1\right)  +\psi\left(  n+k+1\right)  -\psi\left(
v_{2}-u_{1}+k\right)  -\psi\left(  v_{1}-u_{1}+k\right)  \right\}  x^{k}.
\end{gather*}
\bigskip

If $u_{1}-u_{2}\notin\mathbb{Z}$ then near $x=0$ the density is $\sim
C_{0}x^{u_{1}-1}+C_{1}x^{u_{2}-1}$, but if $u_{2}=u_{1}+n$ then the density
$\sim C_{0}x^{u_{1}-1}+C_{1}x^{u_{1}+n-1}\left(  -\log x\right)  $.

\subsection{Example: parametrized family with $m=3$}

Consider the determinant of a random $4\times4$ state, that is, a random (with
the Hilbert-Schmidt metric) positive-definite matrix with trace one. The
moments can be directly computed for the real and complex cases and
incorporated into a family of variables with a parameter. Here the variable is
$256$ times the determinant (to make the range $\left[  0,1\right]  $) and
$\alpha=\frac{1}{2}$ for $\mathbb{R}$, $\alpha=1$ for $\mathbb{C}$, and
$\alpha=2$ for $\mathbb{H}$ (the quaternions). This example is one of the
motivations for the preparation of this exposition. The problem occurred in
Slater's study of the determinant of a partially transposed state in its role
as separability criterion  \cite{S}. 

The moment sequence is
\[
\frac{\left(  1\right)  _{n}\left(  \alpha+1\right)  _{n}\left(
2\alpha+1\right)  _{n}}{\left(  3\alpha+\frac{5}{4}\right)  _{n}\left(
3\alpha+\frac{3}{2}\right)  _{n}\left(  3\alpha+\frac{7}{4}\right)  _{n}%
},n=0,1,2,\cdots;
\]
thus $\delta=6\alpha+\frac{3}{2}$. For generic $\alpha$ the density is%
\begin{align*}
&  \frac{3\left(  12\alpha+1\right)  \left(  6\alpha+1\right)  \left(
4\alpha+1\right)  }{64\alpha^{2}}~_{3}F_{2}\left(
%TCIMACRO{\QATOP{\frac{3}{4}-3\alpha,\frac{1}{2}-3\alpha,\frac{1}{4}-3\alpha
%}{1-\alpha,1-2\alpha}}%
%BeginExpansion
\genfrac{}{}{0pt}{}{\frac{3}{4}-3\alpha,\frac{1}{2}-3\alpha,\frac{1}%
{4}-3\alpha}{1-\alpha,1-2\alpha}%
%EndExpansion
;x\right)  \\
&  -\frac{\Gamma\left(  6\alpha+\frac{5}{2}\right)  \Gamma\left(
3\alpha+\frac{3}{2}\right)  2^{10\alpha}}{4\alpha\sin\left(  \pi\alpha\right)
\Gamma\left(  \alpha+1\right)  \Gamma\left(  8\alpha+1\right)  }x^{\alpha
}~_{3}F_{2}\left(
%TCIMACRO{\QATOP{\frac{3}{4}-2\alpha,\frac{1}{2}-2\alpha,\frac{1}{4}-2\alpha
%}{1-\alpha,1+\alpha}}%
%BeginExpansion
\genfrac{}{}{0pt}{}{\frac{3}{4}-2\alpha,\frac{1}{2}-2\alpha,\frac{1}%
{4}-2\alpha}{1-\alpha,1+\alpha}%
%EndExpansion
;x\right)  \\
&  +\frac{\left(  2\alpha+1\right)  ^{2}\pi^{3}\Gamma\left(  3\alpha+\frac
{5}{2}\right)  \Gamma\left(  6\alpha+\frac{5}{2}\right)  2^{-8\alpha}}%
{48\sin\left(  \pi\alpha\right)  \sin\left(  2\pi\alpha\right)  \Gamma\left(
2\alpha+\frac{1}{2}\right)  \Gamma\left(  \alpha+\frac{3}{2}\right)
^{3}\Gamma\left(  \alpha+1\right)  ^{4}}x^{2\alpha}\\
&  \times~_{3}F_{2}\left(
%TCIMACRO{\QATOP{\frac{3}{4}-\alpha,\frac{1}{2}-\alpha,\frac{1}{4}-\alpha
%}{1+\alpha,1+2\alpha}}%
%BeginExpansion
\genfrac{}{}{0pt}{}{\frac{3}{4}-\alpha,\frac{1}{2}-\alpha,\frac{1}{4}%
-\alpha}{1+\alpha,1+2\alpha}%
%EndExpansion
;x\right)  .
\end{align*}
For numeric computation at $\alpha=\frac{1}{2},1,2$ one can employ
interpolation techniques; for example%
\begin{align*}
f\left(  \alpha_{0};x\right)   &  =\frac{2}{3}\left(  f\left(  \alpha
_{0}+h;x\right)  +f\left(  \alpha_{0}-h;x\right)  \right)  -\frac{1}{6}\left(
f\left(  \alpha_{0}+2h;x\right)  +f\left(  \alpha_{0}-2h;x\right)  \right)  \\
&  -\frac{1}{6}\left(  \frac{\partial}{\partial\alpha}\right)  ^{4}f\left(
\alpha_{0}+\xi h;x\right)  h^{4},
\end{align*}
where $f\left(  \alpha;x\right)  $ denotes the density for specific $\alpha$
and the last term is the error (for some $\xi\in\left(  -2,2\right)  $); thus
the perturbed densities can be computed by the general formula.

\subsection{Example: the recurrence for $m=4$}

Given $u_{1},\ldots,u_{4},v_{1},\ldots,v_{4}$ define $p\left(  c\right)
=\prod_{i=1}^{4}\left(  c+1-u_{i}\right)  ,q\left(  c\right)  =\prod_{i=1}%
^{4}\left(  c+2-v_{i}\right)  $, $q_{1}\left(  c\right)  =(c+1)q\left(
c\right)  -cq\left(  c-1\right)  $, $\delta=\sum_{i=1}^{4}\left(  v_{i}%
-u_{i}\right)  ,$%
\begin{align*}
R_{0}^{\prime}\left(  \gamma\right)   &  =p\left(  \gamma\right)
-q_{1}\left(  \gamma\right)  ,\\
R_{1}^{\prime}\left(  \gamma\right)   &  =\nabla p\left(  \gamma\right)
-\frac{1}{2}\nabla q_{1}\left(  \gamma\right)  ,\\
R_{2}^{\prime}\left(  \gamma\right)   &  =\frac{1}{2}\nabla^{2}p\left(
\gamma\right)  -\frac{1}{6}\nabla^{2}q_{1}\left(  \gamma\right)  ,
\end{align*}
then set $c_{0}=1$,%
\begin{align*}
c_{1} &  =\frac{1}{\delta}R_{2}^{\prime}\left(  \delta-1\right)  c_{0},\\
c_{2} &  =\frac{1}{2\left(  \delta+1\right)  }R_{2}^{\prime}\left(
\delta\right)  c_{1}-\frac{1}{2\delta\left(  \delta+1\right)  }R_{1}^{\prime
}\left(  \delta-1\right)  c_{0},\\
c_{3} &  =\frac{1}{3\left(  \delta+2\right)  }R_{2}^{\prime}\left(
\delta+1\right)  c_{2}-\frac{1}{3\left(  \delta+1\right)  \left(
\delta+2\right)  }R_{1}^{\prime}\left(  \delta\right)  c_{1}\\
&  +\frac{1}{3\delta\left(  \delta+1\right)  \left(  \delta+2\right)  }%
R_{0}^{\prime}\left(  \delta-1\right)  c_{0},\\
c_{n} &  =\frac{1}{n\left(  \delta+n-1\right)  }R_{2}^{\prime}\left(
n+\delta-2\right)  c_{n-1}-\frac{1}{n\left(  \delta+n-2\right)  _{2}}%
R_{1}^{\prime}\left(  n+\delta-3\right)  c_{n-2}\\
&  +\frac{1}{n\left(  \delta+n-3\right)  _{3}}R_{0}^{\prime}\left(
n+\delta-4\right)  c_{n-3}+\frac{1}{n\left(  \delta+n-3\right)  _{3}}q\left(
n+\delta-5\right)  c_{n-4},
\end{align*}
for $n\geq4$.

\subsection{Example: a Macdonald-Mehta-Selberg integral}

Let $S$ be the $3$-dimensional unit sphere $\left\{  x\in\mathbb{R}^{4}%
:\sum_{i=1}^{4}x_{i}^{2}=1\right\}  $ with normalized surface measure
$d\omega$. Consider $\prod_{1\leq i<j\leq4}\left(  x_{i}-x_{j}\right)  ^{2}$
as a random variable (that is, evaluated at a $d\omega$-random point).
Interestingly,  the maximum value $\frac{1}{108}$ is achieved at the 24 points
with (permutations of the) coordinates $\left\{  \pm\frac{1}{6}\sqrt
{9\pm3\sqrt{6}}\right\}  $, which is the zero-set of the rescaled Hermite
polynomial $H_{4}\left(  \sqrt{6}t\right)  $. The Macdonald-Mehta-Selberg
integral (see \cite[p. 319]{DX}) implies (for $\kappa\geq0$)%
\[
\int_{S}\prod_{1\leq i<j\leq4}\left\vert x_{i}-x_{j}\right\vert ^{2\kappa
}d\omega\left(  x\right)  =\frac{1}{2^{6\kappa}}\frac{\Gamma\left(
1+2\kappa\right)  \Gamma\left(  1+3\kappa\right)  \Gamma\left(  1+4\kappa
\right)  }{\Gamma\left(  2+6\kappa\right)  \Gamma\left(  1+\kappa\right)
^{3}}.
\]
For integer values $\kappa=n$ the Gamma functions simplify to Pochhammer
symbols; then by use of formulas like $\left(  1\right)  _{4n}=4^{4n}\left(
\frac{1}{4}\right)  _{n}\left(  \frac{1}{2}\right)  _{n}\left(  \frac{3}%
{4}\right)  _{n}\left(  1\right)  _{n}$ the value becomes%
\[
\mu_{n}=\frac{1}{108^{n}}\frac{\left(  \frac{1}{4}\right)  _{n}\left(
\frac{1}{2}\right)  _{n}\left(  \frac{3}{4}\right)  _{n}}{\left(  \frac{5}%
{6}\right)  _{n}\left(  1\right)  _{n}\left(  \frac{7}{6}\right)  _{n}}.
\]
Let $f_{D}$ denote the density function of $D=108\prod_{1\leq i<j\leq4}\left(
x_{i}-x_{j}\right)  ^{2}$ (by the general results the range of $D$ is $\left[
0,1\right]  $). Applying Theorem \ref{thm1mx} we find%
\[
f_{D}\left(  x\right)  =\frac{\sqrt{2}}{3\pi}\left(  1-x\right)  ^{\frac{1}%
{2}}\left\{  1+\frac{221}{216}\left(  1-x\right)  +\frac{156697}%
{155520}\left(  1-x\right)  ^{2}+\frac{232223093}{235146240}\left(
1-x\right)  ^{3}+\ldots\right\}  .
\]
By formula (\ref{bigf})%
\begin{align*}
f_{D}\left(  x\right)   &  =\gamma_{1}x^{-\frac{3}{4}}~_{3}F_{2}\left(
%TCIMACRO{\QATOP{\frac{5}{12},\frac{1}{4},\frac{1}{12}}{\frac{3}{4},\frac{1}%
%{2}}}%
%BeginExpansion
\genfrac{}{}{0pt}{}{\frac{5}{12},\frac{1}{4},\frac{1}{12}}{\frac{3}{4}%
,\frac{1}{2}}%
%EndExpansion
;x\right)  +\gamma_{2}x^{-\frac{1}{2}}~_{3}F_{2}\left(
%TCIMACRO{\QATOP{\frac{2}{3},\frac{1}{2},\frac{1}{3}}{\frac{3}{4},\frac{5}{4}%
%}}%
%BeginExpansion
\genfrac{}{}{0pt}{}{\frac{2}{3},\frac{1}{2},\frac{1}{3}}{\frac{3}{4},\frac
{5}{4}}%
%EndExpansion
;x\right)  \\
&  +\gamma_{3}x^{-\frac{1}{4}}~_{3}F_{2}\left(
%TCIMACRO{\QATOP{\frac{3}{4},\frac{7}{12},\frac{11}{12}}{\frac{3}{2},\frac
%{5}{4}}}%
%BeginExpansion
\genfrac{}{}{0pt}{}{\frac{3}{4},\frac{7}{12},\frac{11}{12}}{\frac{3}{2}%
,\frac{5}{4}}%
%EndExpansion
;x\right)  ,
\end{align*}
where
\begin{align*}
\gamma_{1} &  =\frac{\pi}{3\Gamma\left(  \frac{3}{4}\right)  ^{2}\Gamma\left(
\frac{7}{12}\right)  \Gamma\left(  \frac{11}{12}\right)  },\\
\gamma_{2} &  =-\frac{2\sqrt{3}}{3\pi},\\
\gamma_{3} &  =\frac{1}{3\pi^{3}}\Gamma\left(  \frac{3}{4}\right)  ^{2}%
\Gamma\left(  \frac{7}{12}\right)  \Gamma\left(  \frac{11}{12}\right)  .
\end{align*}
It is straightforward to derive a series for the cumulative distribution
function $F_{D}\left(  x\right)  =\int_{0}^{x}f_{D}\left(  t\right)  dt$.
Figures \ref{cumul0} and \ref{cumula1} are graphs of $f_{D}$ and $F_{D}$
respectively (of course there is vertical asymptote for $f_{D}$). For
computations we used terms up to the eighth power, with the series in $x$ for
$0<x\leq0.55$ and the $\left(  1-x\right)  $ series for $0.55<x\leq1$. For a
better view there is a graph of $F_{D}\left(  x\right)  $ for $0\leq
x\leq0.04$ in Fig.\ref{cumula2} and of $1-F_{D}\left(  x\right)  $ for
$0.4\leq x\leq1$ in Fig. \ref{cumul3}.%

%TCIMACRO{\FRAME{fpFU}{4.0413in}{3.045in}{0pt}{\Qcb{Density of D, partial
%view}}{\Qlb{cumul0}}{dens221.eps}{\special{ language "Scientific Word";
%type "GRAPHIC";  maintain-aspect-ratio TRUE;  display "USEDEF";
%valid_file "F";  width 4.0413in;  height 3.045in;  depth 0pt;
%original-width 3.9825in;  original-height 2.9922in;  cropleft "0";
%filename 'dens221.eps';file-properties "XNPEU";}}}%
%BeginExpansion
\begin{figure}
[p]
\begin{center}
\includegraphics[
height=3.045in,
width=4.0413in
]%
{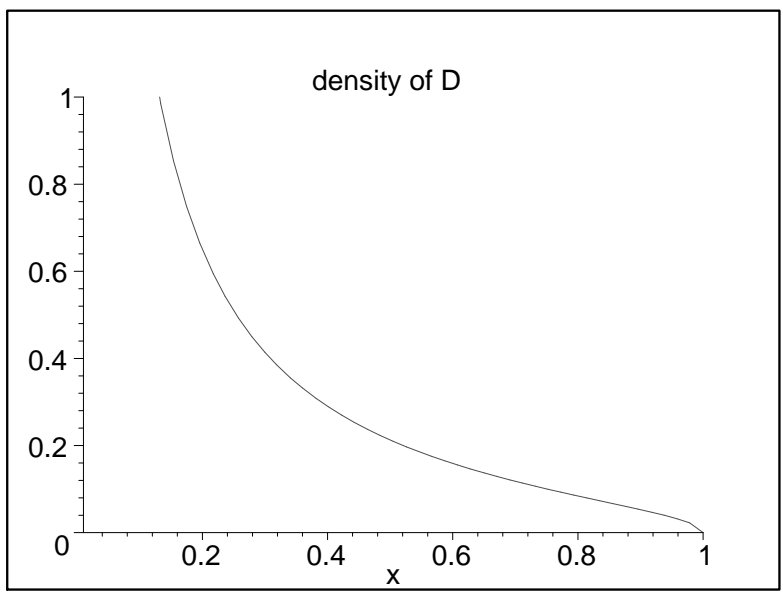}%
\caption{Density of D, partial view}%
\label{cumul0}%
\end{center}
\end{figure}
%EndExpansion
%

%TCIMACRO{\FRAME{ftbpFU}{4.0318in}{3.0364in}{0pt}{\Qcb{Cumulative distribution
%function of D}}{\Qlb{cumula1}}{dens222.eps}%
%width 4.0318in;  height 3.0364in;  depth 0pt;  original-width 3.9825in;
%original-height 2.9922in;  cropleft "0";  croptop "1";  cropright "1";
%cropbottom "0";  filename 'dens222.eps';file-properties "XNPEU";}}}%
%BeginExpansion
\begin{figure}
[ptb]
\begin{center}
\includegraphics[
height=3.0364in,
width=4.0318in
]%
{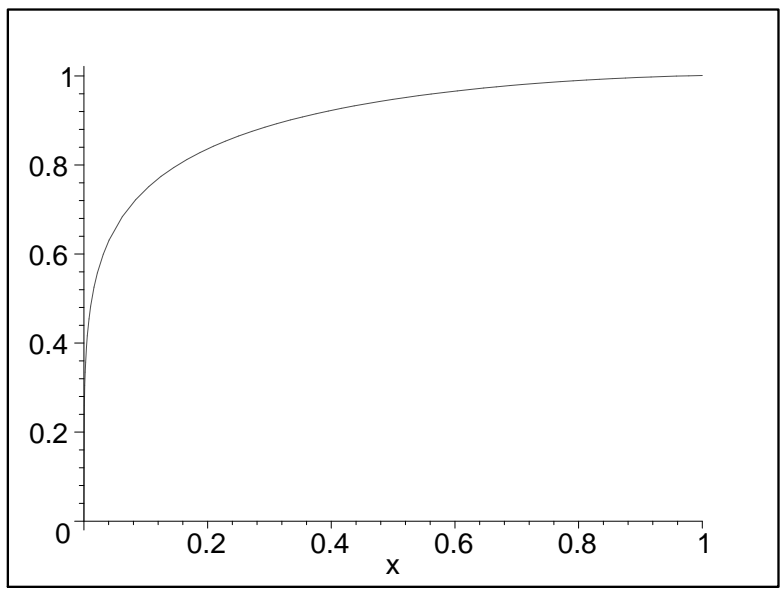}%
\caption{Cumulative distribution function of D}%
\label{cumula1}%
\end{center}
\end{figure}
%EndExpansion
%

%TCIMACRO{\FRAME{ftbpFU}{4.0318in}{3.0364in}{0pt}{\Qcb{Part of cumulative
%distribution of D}}{\Qlb{cumula2}}{dens223.eps}%
%width 4.0318in;  height 3.0364in;  depth 0pt;  original-width 3.9825in;
%original-height 2.9922in;  cropleft "0";  croptop "1";  cropright "1";
%cropbottom "0";  filename 'dens223.eps';file-properties "XNPEU";}}}%
%BeginExpansion
\begin{figure}
[ptb]
\begin{center}
\includegraphics[
height=3.0364in,
width=4.0318in
]%
{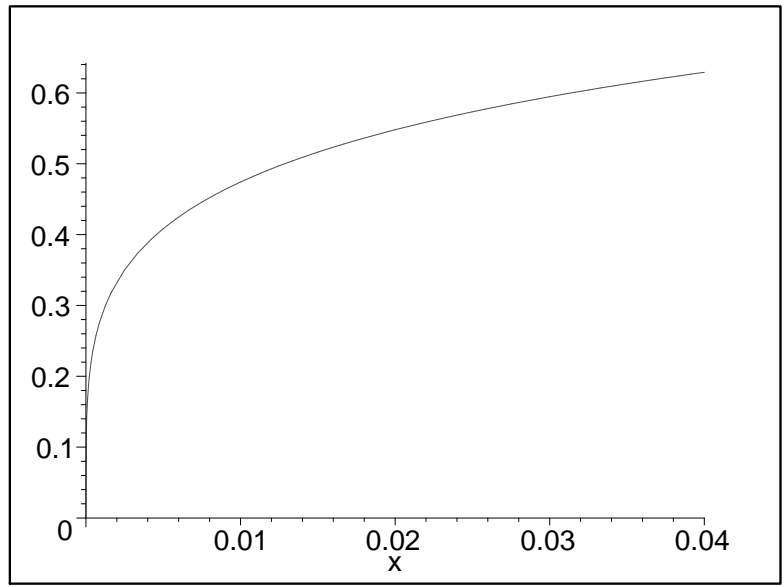}%
\caption{Part of cumulative distribution of D}%
\label{cumula2}%
\end{center}
\end{figure}
%EndExpansion
%

%TCIMACRO{\FRAME{fpFU}{4.0534in}{3.0519in}{0pt}{\Qcb{Part of complementary
%cumulative distribution of D}}{\Qlb{cumul3}}{dens224.eps}%
%width 4.0534in;  height 3.0519in;  depth 0pt;  original-width 3.9825in;
%original-height 2.9922in;  cropleft "0";  croptop "1";  cropright "1";
%cropbottom "0";  filename 'dens224.eps';file-properties "XNPEU";}}}%
%BeginExpansion
\begin{figure}
[p]
\begin{center}
\includegraphics[
height=3.0519in,
width=4.0534in
]%
{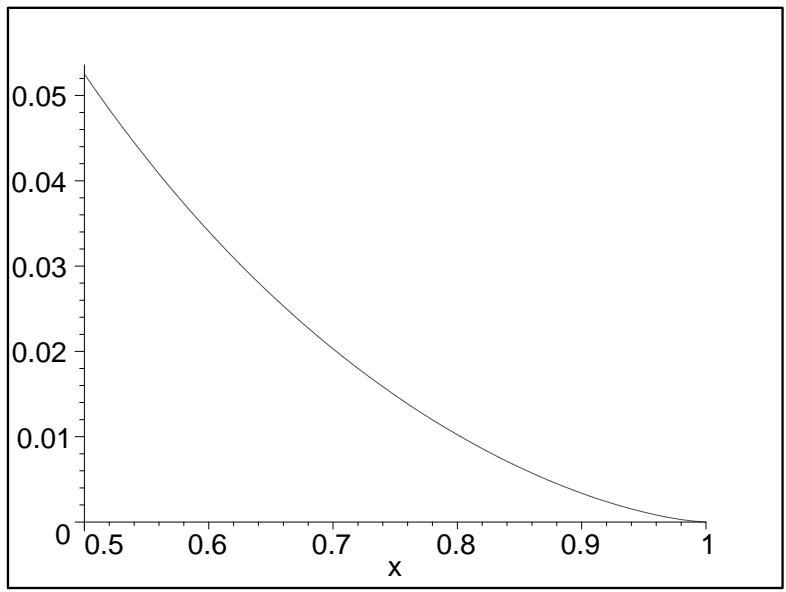}%
\caption{Part of complementary cumulative distribution of D}%
\label{cumul3}%
\end{center}
\end{figure}
%EndExpansion

\end{document}